\documentclass{amsart} 
\usepackage[utf8]{inputenc}
\usepackage{amssymb,amsthm}
\usepackage{mathtools}\mathtoolsset{showonlyrefs}
\usepackage{microtype}
\usepackage{enumitem}
\usepackage{mathrsfs}
\usepackage[bottom=1in]{geometry}
\usepackage{tikz-cd}
\usepackage{csquotes} 
\usepackage{appendix}
\usepackage{hyperref}
\newcommand{\Z}{\mathbb Z}
\newcommand{\Q}{\mathbb Q}

\newcommand{\field}{F}

\newcommand{\ie}{\textit{i.e.}, }

\newcommand{\Hecke}{\mathcal H}
\newcommand{\Mod}[1]{#1\text{-}\mathrm{Mod}}

\newcommand{\ol}[1]{\overline{#1}}
\newcommand{\alg}[1]{\mathbf{#1}}
\newcommand{\set}[2]{\left\{#1\,\middle|\,#2\right\}}

\newcommand{\ideal}[1]{\mathfrak{#1}}
\newcommand{\longtwoheadrightarrow}%
{\relbar\joinrel\twoheadrightarrow}
\newcommand{\longhookrightarrow}%
{\lhook\joinrel\relbar\joinrel\rightarrow}

\DeclareMathOperator{\GL}{GL}

\DeclareMathOperator{\id}{id}

\DeclareMathOperator{\Image}{Im}
\DeclareMathOperator{\Ker}{Ker}

\DeclareMathOperator{\Rad}{Rad}

\newtheorem{prop}{Proposition}

\newtheorem{thm}[prop]{Theorem}

\newtheorem{lem}[prop]{Lemma}
\newtheorem{cor}[prop]{Corollary}
\newtheorem{fact}[prop]{Fact}

\theoremstyle{definition}
\newtheorem*{defn}{Definition}
\newtheorem*{claim*}{Claim}

\theoremstyle{remark}
\newtheorem{rmk}[prop]{Remark}

\title{Localization of the Parabolic Hecke Algebra at a Strictly Positive
Element}
\author{Claudius Heyer}
\address{Mathematisches Institut, Westf\"alische Wilhelms-Universit\"at
M\"unster, Einsteinstra\ss{}e 62, D-48149 M\"unster, Germany}
\email{cheyer@uni-muenster.de}

\begin{document}
\begin{abstract} 
Let $\mathbf{P}$ be a parabolic subgroup with Levi $\mathbf{M}$ of a connected
reductive group defined over a locally compact non-archimedean field $F$. Given
a certain compact open subgroup $\Gamma$ of $\mathbf{P}(F)$, this note proves
that the Hecke algebra $\mathcal{H}(\mathbf{M}(F))$ of $\mathbf{M}(F)$ with
respect to $\Gamma\cap \mathbf{M}(F)$ is a left ring of fractions of the Hecke
algebra $\mathcal{H}(\mathbf{P}(F))$ of $\mathbf{P}(F)$ with respect to
$\Gamma$. This leads to a characterization of
$\mathcal{H}(\mathbf{P}(F))$-modules that come from
$\mathcal{H}(\mathbf{M}(F))$-modules.
\end{abstract} 
\maketitle

\section{Introduction}
\subsection{Motivation} 
Let $\field$ be a locally compact non-archimedean field. Let $\alg P$ be a
parabolic $\field$-subgroup of a connected reductive $\field$-group $\alg G$.
Denote $\alg U$ the unipotent radical and $\alg M$ a Levi subgroup of $\alg P$.
For smooth representations of $\alg P(\field)$ on $k$-vector spaces, for a field
$k$, the following fact\footnote{I learned it from Peter Schneider who observed
it from \cite[Lemma I.2.1]{Harris-Taylor.2001} which gives a rather complicated
proof for a special case.
}
seems to be well-known:
\begin{fact}\label{fact} 
If $\pi$ is a smooth representation of $\alg P(\field)$ whose
restriction to $\alg M(\field)$ is admissible, then $\alg U(\field)$ acts
trivially on $\pi$.
\end{fact} 

For the convenience of the reader we give an easy proof in an appendix. 

Let $I_1$ be a pro-$p$ Iwahori subgroup of $\alg G(\field)$ which is compatible
with $\alg P$. If $k$ has characteristic equal to the residue characteristic of
$\field$, then the functor $\pi\mapsto \pi^{I_1}$ from the category of smooth
representations of $\alg G(\field)$ to the category of modules over the
pro-$p$ Iwahori--Hecke $k$-algebra $H_k(I_1,\alg G(\field))$ of
$I_1$-biinvariant functions $\alg G(F)\to k$ is of central importance. In
\cite{Heyer.2020} the parabolic induction functor for Hecke modules, see
\cite{Ollivier-Vigneras.2018} or \cite{Vigneras.2015}, was realized by
constructing $k$-algebra homomorphisms
\[
\begin{tikzcd}
& H_k\bigl(I_1\cap \alg P(\field), \alg P(\field)\bigr) \ar[dl,"\Theta^P_M"']
\ar[dr,"\Xi^P_G"]\\
H_k\bigl(I_1\cap \alg M(\field), \alg M(\field)\bigr) & & H_k\bigl(I_1, \alg
G(\field)\bigr).
\end{tikzcd}
\]
It is then a natural question to ask whether Fact~\ref{fact} 
has a counterpart for $H_k\bigl(I_1\cap \alg P(\field), \alg
P(\field)\bigr)$-modules. More concretely, we ask: does every left
$H_k\bigl(I_1\cap \alg P(\field),\alg P(\field)\bigr)$-module arise from a
left $H_k\bigl(I_1\cap \alg M(\field),\alg M(\field)\bigr)$-module by
restricting along $\Theta^P_M$? The answer to the question as stated turns out
to be negative. In this note we characterize the $H_k\bigl(I_1\cap \alg
P(\field), \alg P(\field)\bigr)$-modules that arise by restriction along
$\Theta^P_M$. See Corollary~\ref{cor:modules} for the precise statement. In
fact, we prove (Theorem~\ref{thm:local}) that $H_k\bigl(I_1\cap \alg M(\field),
\alg M(\field)\bigr)$ is the localization of $H_k\bigl(I_1\cap \alg P(\field),
\alg P(\field)\bigr)$ at a certain element $T_a$ associated with a strictly
positive element $a$ of $\alg M(\field)$. 

If $\alg B_2$ is the Borel subgroup of upper triangular matrices in $\GL_2$,
and $\field = \Q_p$, the smooth modulo $p$ representations of $\alg B_2(\Q_p)$
and their relation with smooth $\GL_2(\Q_p)$-representations have
been studied in \cite{Berger.2010} and \cite{Vienney.2012}. In this light it may
be desirable to better understand $H_k\bigl(I_1\cap \alg B_2(\Q_p),\alg
B_2(\Q_p)\bigr)$-modules (maybe with $I_1$ replaced by some other compact open
subgroup). The results in this note are a first step in this direction.
\subsection{Acknowledgments} 
The results of this note were obtained in my doctoral dissertation
\cite{Heyer.2019} at the Humboldt University in Berlin. It is a pleasure to
thank my advisor Elmar Gro\ss{}e-Kl\"onne for asking the question above and for
his support. I also thank Peter Schneider for his interest in my work. During
the write-up of this note I was funded by the University of M\"unster and
Germany’s Excellence Strategy EXC 2044 390685587, Mathematics M\"unster:
Dynamics--Geometry--Structure.

\section{Reminder on abstract Hecke algebras}\label{sec:Hecke} 
We recall the notion of an abstract Hecke algebra. Details can be found in
\cite[Chapter~3, \S1.2]{Andrianov.1995}.
Let $G$ be a topological group and $S\subseteq G$ a submonoid containing
an open compact subgroup $\Gamma$. Then $(\Gamma g)g' \coloneqq (\Gamma gg')$
induces a right action of $S$ on the free $\Z$-module $\Z[\Gamma\backslash S]$
generated by the right cosets $(\Gamma g)$ of $\Gamma$ in $S$. Given $g\in S$,
we put
\[
T_g \coloneqq T_g^S \coloneqq \sum_{\Gamma h\subseteq \Gamma g\Gamma}(\Gamma
h),
\]
where the sum is finite and runs over all right cosets contained in $\Gamma
g\Gamma$. It is clear that $T_g$ depends only on the double coset $\Gamma
g\Gamma$. The following facts are well-known:
\begin{enumerate}[label=\roman*.]
\item\label{sec:Hecke-i} The $\Z$-module $H(\Gamma, S) \coloneqq
\Z[\Gamma\backslash S]^\Gamma$ of $\Gamma$-invariants is free with $\Z$-basis
$(T_g)_{\Gamma g\Gamma \in \Gamma\backslash S/G}$.
\item\label{sec:Hecke-ii} The multiplication $\sum_i a_i(\Gamma g_i)\cdot
\sum_jb_j(\Gamma h_j)\coloneqq \sum_{i,j} a_ib_j(\Gamma g_ih_j)$
turns $H(\Gamma, S)$ into an associative ring with unit $T_1$.
\end{enumerate}

Given a commutative ring $R$, we call
\[
H_R(\Gamma, S) \coloneqq R\otimes_\Z H(\Gamma, S)
\]
the \emph{Hecke algebra} over $R$ associated with $(\Gamma, S)$.
\section{The parabolic Hecke algebra}\label{sec:parabolic} 
Let $P = U\rtimes M$ be a semidirect product of locally profinite groups.
Let $\Gamma$ be a compact open subgroup of $P$ such that $\Gamma =
\Gamma_U\Gamma_M$, where $\Gamma_U \coloneqq \Gamma\cap U$ and $\Gamma_M
\coloneqq \Gamma\cap M$. We fix a commutative ring $R$.

\subsection{Positive elements}\label{subsec:positive} 
\begin{defn} 
\begin{enumerate}[label=(\alph*)]
\item An element $m\in M$ is called \emph{positive} if $m\Gamma_U m^{-1}
\subseteq \Gamma_U$. The positive elements form a monoid, denoted $M^+$. Note
that $\Gamma_M\subseteq M^+$.

\item A positive element $a$ in the center of $M$ is called \emph{strictly
positive} if $\bigcup_{n>0} a^{-n}\Gamma_Ua^n = U$. Note that the existence of
strictly positive elements necessitates that $U$ be the union of its compact
open subgroups.
\end{enumerate}
\end{defn} 

\begin{rmk}\label{rmk:parabolic} 
The setup one should have in mind is the following: let $\alg G$ be a connected
reductive group over a locally compact non-archimedean field $\field$. Let $\alg
P$ be a parabolic subgroup with unipotent radical $\alg U$ and Levi subgroup
$\alg M$. Denote $\ol{\alg P}$ the parabolic subgroup opposite to $\alg P$, \ie
the unique parabolic with Levi $\alg M$ satisfying $\ol{\alg P}\cap \alg P
= \alg M$. Let $I$ be a compact open subgroup of $\alg G(\field)$ which
decomposes as
\[
I = (I\cap \ol{\alg U}(\field))(I\cap \alg M(\field))(I\cap \alg U(\field)),
\] 
where $\ol{\alg U}$ is the unipotent radical of $\ol{\alg P}$. In this setting
positive elements were introduced and studied in
\cite[(6.5)]{Bushnell-Kutzko.1998} and \cite[II.4]{Vigneras.1998}. There, an
element $m\in \alg M(\field)$ is called positive if it is positive in our
sense\footnote{Here, $P = \alg P(\field)$, $M = \alg M(\field)$, $U = \alg
U(\field)$, and $\Gamma = P\cap I$.}
and if, in addition, $m^{-1}(I\cap \ol{\alg U}(\field))m \subseteq I\cap
\ol{\alg U}(\field)$. If $I$ is a (pro-$p$) Iwahori subgroup compatible with
$\alg P$ (meaning that $\alg P$ contains the minimal parabolic
subgroup determined by $I$), then it was proved in \cite[\S3.5]{Heyer.2020} that
both notions agree. Observe that our setting also allows for $I$ to be a
certain maximal compact or special parahoric subgroup of $G$.

We also remark that our definition of ``strictly positive'' is nonstandard and
more general than the definitions in \cite[(6.16)]{Bushnell-Kutzko.1998} (who
use the term ``strongly positive'') and \cite[II.4]{Vigneras.1998}.
\end{rmk} 

The following trivial lemma explains the role of strictly positive elements.
\begin{lem}\label{lem:strictly positive} 
Let $a\in M$ be a strictly positive element. For every $m\in M$ there exists
$n\in \Z_{>0}$ such that $a^nm = ma^n \in M^+$.
\end{lem} 
\begin{proof} 
Since $m\Gamma_Um^{-1}$ is compact and $a$ is strictly positive, there exists
$n>0$ such that $m\Gamma_Um^{-1} \subseteq a^{-n}\Gamma_Ua^n$. This means
$a^nm\in M^+$.
\end{proof} 
\subsection{The centralizer of a strictly positive 
element}\label{subsec:centralizer}
Consider the Hecke algebra
\[
\Hecke_R(P)\coloneqq H_R(\Gamma, P).
\]
Let $a\in M$ be a strictly positive element and define the centralizer
$R$-algebra
\[
C_{\Hecke_R(P)}(a) \coloneqq \set{X\in \Hecke_R(P)}{XT_a = T_aX}.
\]
If $P$ is the Siegel parabolic of a symplectic group the algebra
$C_{\Hecke_R(P)}(a)$ was already studied by Andrianov. The results of this
paragraph have their origins in \cite[\S4]{Andrianov.1977}.

\begin{lem}\label{lem:C(a)} 
\begin{enumerate}[label=(\alph*)]
\item\label{lem:C(a)-a} One has $C_{\Hecke_R(P)}(a) = \set{X\in \Hecke_R(P)}{X
= \sum_{i=1}^r\lambda_i\cdot (\Gamma m_i), \text{with $m_i\in M$}}$.
\item\label{lem:C(a)-b} The $R$-algebra $C_{\Hecke_R(P)}(a)$ is a free
$R$-module with basis $\mathcal B \coloneqq \set{T_m \in \Hecke_R(P)}{m\in
M^+}$.
\item\label{lem:C(a)-c} Given $X\in \Hecke_R(P)$, there exists $n>0$ such
that $T_a^nX \in C_{\Hecke_R(P)}(a)$.
\end{enumerate}
\end{lem} 
\begin{proof} 
Given $m\in M^+$, we have $\Gamma m\Gamma = \Gamma m\Gamma_M$, since $\Gamma =
\Gamma_U\Gamma_M$ and $m\Gamma_Um^{-1} \subseteq \Gamma_U$. Since $a$ is central
and positive, we deduce $\Gamma a\Gamma = \Gamma a$, whence $T_a = (\Gamma a)$
in $\Hecke_R(P)$. Clearly, every $X \in \Hecke_R(P)$ of the form $X =
\sum_{i=1}^r \lambda_i\cdot (\Gamma m_i)$, with $m_i\in M$, lies in
$C_{\Hecke_R(P)}(a)$. Let $X = \sum_{i=1}^r \lambda_i\cdot (\Gamma u_im_i) \in
\Hecke_R(P)$, where $u_i\in U$, $m_i\in M$. As $a$ is strictly positive, there
exists $n>0$ such that $u_1,\dotsc,u_r\in a^{-n}\Gamma_Ua^n$. Therefore,
\[
T_a^nX = \sum_{i=1}^r \lambda_i\cdot (\Gamma a^nu_ia^{-n}\cdot a^nm_i) =
\sum_{i=1}^r\lambda_i\cdot (\Gamma a^nm_i) \in C_{\Hecke_R(P)}(a).
\]
This proves \ref{lem:C(a)-c}. If $X$ as above already lies in
$C_{\Hecke_R(P)}(a)$, then, using the right action of $P$ on $R[\Gamma\backslash
P]$, the same computation shows $X\cdot a^n = XT_a^n = T_a^nX = \sum_{i=1}^r
\lambda_i\cdot (\Gamma m_i)\cdot a^n$. This proves \ref{lem:C(a)-a}. Finally,
let $C$ be the free $R$-module generated by $\mathcal B$. From \ref{lem:C(a)-a}
it follows that $C\subseteq C_{\Hecke_R(P)}(a)$. Conversely, let $X\in
C_{\Hecke_R(P)}(a)$ and write $X = \sum_{i=1}^r \lambda_i\cdot (\Gamma m_i)$,
with $m_i\in M$. It remains to show $m_1,\dotsc,m_r\in M^+$. Let $u\in \Gamma_U$
be arbitrary. Since $X$ is $\Gamma$-invariant, there exists a permutation
$\sigma_u$ on $\{1,\dotsc,r\}$ such that $\Gamma m_iu = \Gamma m_{\sigma_u(i)}$.
Hence, there exists $y\in \Gamma$ with $m_ium_i^{-1}\cdot m_i = m_iu =
ym_{\sigma_u(i)}$. Writing $y = y_Uy_M$, where $y_U\in \Gamma_U$ and $y_M\in
\Gamma_M$, and projecting to $U$ yields $m_ium_i^{-1} = y_U\in \Gamma_U$. As $u$
was arbitrary, this shows that $m_i$ is positive. Now, $X\in C$ follows from
\ref{sec:Hecke-i} in section~\ref{sec:Hecke}.
\end{proof} 

Put $\Hecke_R(M) \coloneqq H_R(\Gamma_M,M)$. In \cite{Heyer.2020} the
$R$-algebra homomorphism
\[
\Theta^P_M\colon \Hecke_R(P) \longrightarrow \Hecke_R(M),\qquad
\sum_{i=1}^r\lambda_i\cdot (\Gamma u_im_i) \longmapsto \sum_{i=1}^r
\lambda_i\cdot (\Gamma_Mm_i)
\]
was studied. Recall that $\Hecke_R(M)$ is a localization of the subalgebra
$\Hecke_R(M^+)\coloneqq H_R(\Gamma_M,M^+)$ at the element $T_a^{M^+}$
\cite[II.6]{Vigneras.1998}. 

\begin{cor}\label{cor:C(a)} 
The map $\Theta^P_M$ induces by restriction an isomorphism $C_{\Hecke_R(P)}(a)
\cong \Hecke_R(M^+)$.
\end{cor} 
\begin{proof} 
$\Theta^P_M$ identifies the basis $\mathcal B$ in
Lemma~\ref{lem:C(a)} \ref{lem:C(a)-b} with the double coset basis of
$\Hecke_R(M^+)$.
\end{proof} 
\subsection{Localization}\label{subsec:local} 
Fix a strictly positive element $a$ of $M$. 
\begin{defn} 
Let $\ideal m$ be a left $\Hecke_R(P)$-module. Define the $R$-submodule
\[
\Rad_a(\ideal m) \coloneqq \set{x\in \ideal m}{T_a^nx = 0\text{ for some
$n>0$}}.
\]
\end{defn} 

\begin{lem}\label{lem:Rad} 
Let $\ideal m$ be a left $\Hecke_R(P)$-module.
\begin{enumerate}[label=(\alph*)]
\item\label{lem:Rad-a} The $R$-submodule $\Rad_a(\ideal m)$ does not depend on
the choice of $a$. We therefore omit the subscript and write simply $\Rad(\ideal
m)$.

\item\label{lem:Rad-b} $\Rad(\ideal m)$ is an $\Hecke_R(P)$-submodule of $\ideal
m$.
\end{enumerate}
\end{lem} 
\begin{proof} 
Let $b$ be another strictly positive element of $M$. By symmetry it suffices to
show $\Rad_a(\ideal m) \subseteq \Rad_b(\ideal m)$. So let $x\in \Rad_a(\ideal
m)$ and $n>0$ with $T_a^nx = 0$. By Lemma~\ref{lem:strictly positive}, and
because $b$ is strictly positive, there exists $l>0$ such that $m\coloneqq
b^la^{-n}$ is positive. Note also that $m$ lies in the center of $M$. Therefore,
$T_m = (\Gamma m)$ (see the proof of Lemma~\ref{lem:C(a)}) and $T_b^lx =
T_mT_a^nx = 0$. Whence $x\in \Rad_b(\ideal m)$. This proves \ref{lem:Rad-a},
and \ref{lem:Rad-b} is an immediate consequence of
Lemma~\ref{lem:C(a)} \ref{lem:C(a)-c}.
\end{proof} 

Consider the multiplicatively closed set $\mathcal S_a \coloneqq \set{T_a^n \in
\Hecke_R(P)}{n\in\Z_{\ge0}}$ in $\Hecke_R(P)$. 

\begin{thm}\label{thm:local} 
The map $\Theta^P_M\colon \Hecke_R(P)\to \Hecke_R(M)$ exhibits $\Hecke_R(M)$ as
the left ring of fractions of $\Hecke_R(P)$ with respect to $\mathcal S_a$.
\end{thm} 
\begin{proof} 
According to \cite[(10.1) Definition]{Lam.1999} we need to check the following
conditions:
\begin{enumerate}[label=(\roman*)]
\item\label{thm:local-i} $\Theta^P_M$ is $\mathcal S_a$-inverting, meaning that
$\Theta^P_M(s)$ is invertible in $\Hecke_R(M)$, for every $s\in \mathcal S_a$.
\item\label{thm:local-ii} Every element of $\Hecke_R(M)$ has the form
$\Theta^P_M(s)^{-1}\Theta^P_M(X)$, for some $X\in \Hecke_R(P)$ and $s\in
\mathcal S_a$.
\item\label{thm:local-iii} $\Ker\Theta^P_M = \Rad(\Hecke_R(P))$.
\end{enumerate}
For every $n\ge0$ the element $\Theta^P_M(T_a^n) = \Theta^P_M(T_{a^n}) =
T^M_{a^n}$ is invertible in $\Hecke_R(M)$ with inverse $T^M_{a^{-n}}$, whence
\ref{thm:local-i}. In particular, $\Rad(\Hecke_R(P))$ is contained in
$\Ker\Theta^P_M$. Let $X\in \Ker\Theta^P_M$ be arbitrary. By
Lemma~\ref{lem:C(a)} \ref{lem:C(a)-c} there exists $n>0$ with $T_a^nX \in
C_{\Hecke_R(P)}(a)$. By Corollary~\ref{cor:C(a)} the restriction of
$\Theta^P_M$ to $C_{\Hecke_R(P)}(a)$ is injective. Hence $T_a^nX = 0$, that
is, $X\in \Rad(\Hecke_R(P))$. This shows \ref{thm:local-iii}. Now,
\ref{thm:local-ii} follows from the fact that the image of $\Theta^P_M$ contains
$\Hecke_R(M^+)$ (Corollary~\ref{cor:C(a)}) and that $\Hecke_R(M)$ is a
localization of $\Hecke_R(M^+)$ at $\Theta^P_M(T_a)$
\cite[II.6]{Vigneras.1998}.
\end{proof} 

\begin{cor}\label{cor:modules} 
Let $\ideal m$ be a left $\Hecke_R(P)$-module. Then $\ideal m$ comes from an
$\Hecke_R(M)$-module by restriction along $\Theta^P_M\colon \Hecke_R(P)\to
\Hecke_R(M)$ if and only if $T_a$ acts
as an $R$-linear automorphism on $\ideal m$.
\end{cor} 
\begin{proof} 
Clear from Theorem~\ref{thm:local}.
\end{proof} 

Consider the ring $\mathcal A = \Image \Theta^P_{M,\Z}$, where $\Theta^P_{M,\Z}$
is the map $\Theta^P_M$ for $R=\Z$. Put $\mathcal A_R \coloneqq R\otimes_\Z
\mathcal A$. Then $\Theta^P_M$ factors as
\[
\begin{tikzcd}
\Hecke_R(P) \ar[r,two heads,"\pi"] \ar[dr,"\Theta^P_M"'] & \mathcal A_R
\ar[d,"\theta"]\\ & \Hecke_R(M),
\end{tikzcd}
\]
where $\pi = \id_R\otimes\Theta^P_{M,\Z}$. The algebra $\mathcal A_R$ was
denoted $\Hecke_R(M,G)$ and studied in \cite[section~5]{Heyer.2020} in the
context of Remark~\ref{rmk:parabolic}. Note that Theorem~\ref{thm:local} implies
that $\theta$ is a localization map.

If $S$ is a (possibly noncommutative) ring, we denote $\Mod S$ the category of
left $S$-modules.

\begin{prop}\label{prop:functors} 
There are fully faithful embeddings
\[
\begin{tikzcd}
\Mod{\Hecke_R(M)} \ar[r,hook,"\theta_*"] & \Mod{\mathcal A_R}
\ar[r,hook,"\pi_*"] & \Mod{\Hecke_R(P)}.
\end{tikzcd}
\]
Then $\Mod{\Hecke_R(M)}$ identifies with the full subcategory of those
$\Hecke_R(P)$-modules on which $T_a$ acts invertibly. The category
$\Mod{\mathcal A_R}$ contains all $\Hecke_R(P)$-modules $\ideal m$ with
$\Rad(\ideal m) = \{0\}$.
\end{prop} 
\begin{proof} 
The fully faithfulness claim follows from the fact that $\pi$ is surjective and
$\theta$ is a localization map. The characterization of $\Mod{\Hecke_R(M)}$ is
the statement of Corollary~\ref{cor:modules}. Let $\ideal m$ be a left
$\Hecke_R(P)$-module. Note that $\Ker\pi \subseteq
\Ker\Theta^P_M = \Rad(\Hecke_R(P))$ and $\Rad(\Hecke_R(P)).\ideal m \subseteq
\Rad(\ideal m)$. Hence, $\Ker(\pi).\ideal m \subseteq \Rad(\ideal m)$. It
follows that if $\Rad(\ideal m) = \{0\}$, then $\ideal m$ lies in the essential
image of $\pi_*$.
\end{proof} 

\begin{rmk} 
The map $\theta\colon \mathcal A_R \to \Hecke_R(M)$ is not injective in general.
In this case, $\Rad(\mathcal A_R)\neq \{0\}$ which shows that the essential
image of $\pi_*$ contains $\Hecke_R(P)$-modules $\ideal m$ with $\Rad(\ideal
m)\neq \{0\}$. However, if $\theta$ is injective, then $\Mod{\mathcal A_R}$
identifies with the full subcategory of $\Mod{\Hecke_R(P)}$ consisting of all
$\ideal m$ with $\Rad(\ideal m) = \{0\}$.
\end{rmk} 

\appendix
\section{}
We now give a proof of Fact~\ref{fact} which was stated at the beginning of the
introduction.

Let $\alg P = \alg U\alg M$ be a parabolic subgroup in a connected reductive
$\field$-group, where $\field$ is a locally compact non-archimedean field.
Denote $Z$ the center of the Levi of $\alg M(\field)$. As every admissible
smooth $\alg M(\field)$-representation $\pi$ is locally $Z$-finite (meaning that
every vector in $\pi$ is contained in a finite-dimensional $Z$-invariant
subspace), Fact~\ref{fact} is a consequence of:

\begin{lem} 
Let $\pi$ be a smooth $\alg P(\field)$-representation which is locally
$Z$-finite. Then $\alg U(\field)$ acts trivially on $\pi$.
\end{lem} 
\begin{proof} 
Fix a strictly positive element $a\in Z$ in the sense of
\cite[II.4]{Vigneras.1998}. Let $v\in \pi$ be arbitrary and fix a
finite-dimensional $Z$-invariant subvector space $V$ of $\pi$ containing $v$.
Since $\pi$ is smooth, there exists an open subgroup $H$ of $\alg P(\field)$
fixing $V$ elementwise. Given any $u\in \alg U(\field)$, we find $n\in\Z_{>0}$
with $a^nua^{-n}\in H\cap \alg U(\field)$. Now, $a^nv\in V$ and hence $uv =
a^{-n}\cdot (a^nua^{-n})\cdot a^nv = a^{-n}\cdot a^nv = v$.
\end{proof} 

\bibliographystyle{alphaurl}
\bibliography{references}{}
\end{document}